\documentclass{amsart}

\usepackage[utf8]{inputenc}
\usepackage{amsfonts, amssymb, amsmath, amsthm}
\usepackage{graphicx}
\usepackage{float, comment}
\usepackage[hidelinks]{hyperref}
\usepackage{verbatim}
\usepackage{mathtools, nccmath}

\usepackage{enumitem}

\newtheorem{teor}{Theorem}[section]
\newtheorem{lema}[teor]{Lemma}
\newtheorem{prop}[teor]{Proposition}
\newtheorem{fact}[teor]{Fact}

\newtheorem{defin}[teor]{Definition}
\newtheorem*{claim}{Claim}
\newtheorem{cor}[teor]{Corollary}
\newtheorem{remark}[teor]{Remark}
\newtheorem{problem}[teor]{Problem}

\newcommand{\bigslant}[2]{{\raisebox{.2em}{$#1$}\left/\raisebox{-.2em}{$#2$}\right.}}

\newcommand{\C}{{\mathfrak C}}
\newcommand{\R}{\mathbb{R}}

\DeclareMathOperator{\tp}{{tp}}
\DeclareMathOperator{\Th}{{Th}}
\DeclareMathOperator{\aut}{{Aut}}
\DeclareMathOperator{\med}{{med}}
\DeclareMathOperator{\EM}{{EM}}
\DeclareMathOperator{\Lin}{{Lin}}
\DeclareMathOperator\dcl{dcl}

\newcommand{\orcidlogo}{\includegraphics[height=\fontcharht\font`\B]{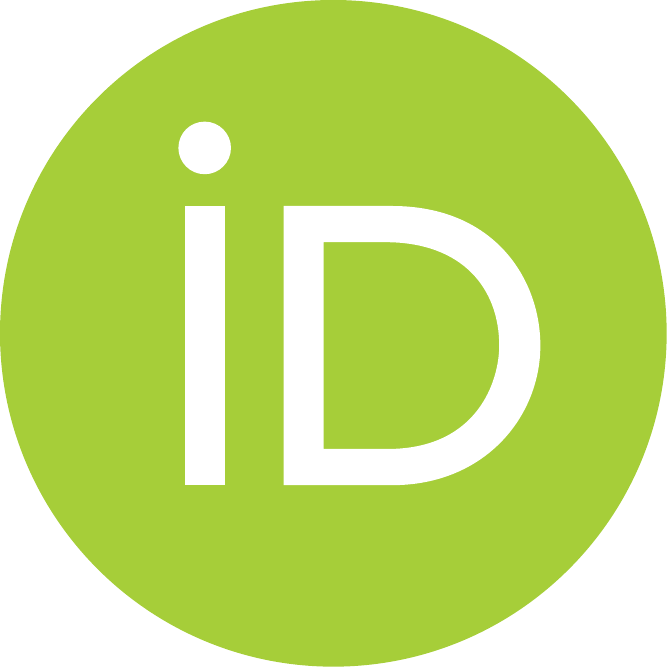}}
\newcommand{\orcid}[1]{\href{#1}{\orcidlogo #1}}

\title{On stable quotients}
\author{Krzysztof Krupi\'{n}ski}
\author{Adri\'{a}n Portillo}

\thanks{\noindent Both authors are supported by the Narodowe Centrum Nauki grant no. 2016/22/E/ST1/00450. The first author is also supported by  the Narodowe Centrum Nauki grant no. 2018/31/B/ST1/00357.}

\address{Instytut Matematyczny Uniwersytetu Wroc{\l}awskiego, pl. Grunwaldzki 2, 50-384 Wroc{\l}aw, Poland}
\address{Krzysztof Krupi\'{n}ski \orcid{https://orcid.org/0000-0002-2243-4411}}
\email{Krzysztof.Krupinski@math.uni.wroc.pl}
\address{Adri\'{a}n Portillo \orcid{https://orcid.org/0000-0001-9354-8574}}
\email{Adrian.Portillo-Fernandez@math.uni.wroc.pl}


\usepackage{filecontents}

\usepackage[
backend=bibtex,
style=alphabetic,
]{biblatex}

\usepackage{stackengine}
\stackMath
\DeclareMathOperator*\dcap{{\stackinset{r}{-1ex}{c}{-1.9pt}{\downarrow}
		{\bigcap}\mkern2mu}}
\DeclareMathOperator*\acup{{\stackinset{r}{-1ex}{c}{1.9pt}{\uparrow}
		{\bigcup}\mkern2mu}}

\begin{document}
	\maketitle

	\begin{abstract}
We solve two problems from \cite{MR3796277} concerning maximal stable quotients of groups $\bigwedge$-definable in NIP theories. The first result says that if $G$ is a $\bigwedge$-definable group in a distal theory, then $G^{st}=G^{00}$ (where $G^{st}$ is the smallest $\bigwedge$-definable subgroup with $G/G^{st}$ stable, and $G^{00}$ is the smallest $\bigwedge$-definable subgroup of bounded index). In order to get it, we prove that distality is preserved under passing from $T$ to the hyperimaginary expansion $T^{heq}$. The second result is an example of a group $G$ definable in a non-distal, NIP theory for which $G=G^{00}$ but $G^{st}$ is not an intersection of definable groups. Our example is a saturated extension of $(\mathbb{R},+,[0,1])$. Moreover, we make some observations on the question whether there is such an example which is a group of finite exponent. We also take the opportunity and give several characterizations of stability of hyperdefinable sets, involving continuous logic.
	\end{abstract}

	\section{Introduction}
	
The core of model-theory is stability theory, developed in the 70's and 80's of the previous century. In the past three decades, one of the main goals of model theory has become finding extensions of stability theory to various unstable contexts, covering many mathematically interesting examples. One either tries to impose some general global assumptions on the theory in question (such as NIP or simplicity) or some local ones (e.g work with a stable definable set or generically stable type), and prove some structural results. 
It is also natural and ubiquitous in model theory to look at a ``global-local'' situation, namely quotients by $\bigwedge$-definable (i.e. type-definable) equivalence relations and assume (or prove) their good properties (e.g. boundedness) to get some further conclusions. Recall that a {\em hyperimaginary} is a class of a $\bigwedge$-definable equivalence relation, and a {\em hyperdefinable set} is a quotient of a $\bigwedge$-definable set by such a relation. While bounded quotients have played an important role in model theory and its applications (e.g. to approximate subgroups) for many years, stable quotients have not been studied thoroughly. They  appeared in \cite{MR3796277}. 
However, it is folklore that hyperimaginaries can be treated as imaginaries in continuous logic via a definable pseudometric (see \cite{2010,conant2021separation} and \cite[Chapter 3]{hanson2020thesis} in the language of continuous logic and \cite{10.2178/jsl/1122038916} in the language of CATs), so in a sense stability of hyperdefinable sets is equivalent to stability (of imaginary sorts) in continuous logic developed in \cite{MR2657678,MR2723787}. This is an additional motivation to consider stable quotients. So in Section \ref{section: characterizations of stability} we take the opportunity and give several characterizations of stability of hyperdefinable sets in various terms, involving continuous logic, including generically stable types which seem to be not considered (or even defined) so far in this context.

In the main parts of the paper, however, we will study stability of hyperdefinable sets (mostly groups)  without referring to continuous logic, just using the definition from \cite{MR3796277} which we recall below, or a characterization via bounds on the number of types observed in Section \ref{section: characterizations of stability}. 

Let $T$ be a complete theory, $\C \models T$ a monster (i.e. $\kappa$-saturated and strongly $\kappa$-homogeneous for a strong limit cardinal $> |T|$) model  in which we are working, and $A \subset \C$ a small set of parameters (i.e. $|A| < \kappa$); a cardinal $\gamma$ is {\em bounded} if  $\gamma < \kappa$.

		\begin{defin}
		A hyperdefinable (over $A$) set $X/E$ is \emph{stable} if for every $A$-indiscernible sequence $(a_i,b_i)_{i<\omega}$ with $a_i\in X/E$ for all (equivalently, some) $i<\omega$, we have  
		$$\tp(a_i,b_j/A) = \tp(a_j,b_i/A)$$ for all (some) $i\neq j < \omega$.
	\end{defin}

Let $G$ be a $\emptyset$-$\bigwedge$-definable group. There is always a smallest $A$-$\bigwedge$-definable subgroup of $G$ of bounded index, which is denoted by $G^{00}_A$. Under NIP, the group $G_A^{00}$ does not depend on the choice of $A$ (see \cite{MR2361885}) and is denoted by $G^{00}$. So $G^{00}$ is the smallest $\bigwedge$-definable (over parameters) subgroup of $G$ of bounded index, and it is in fact $\emptyset$-$\bigwedge$-definable and normal. 
Staying in the NIP context, $G^0$ is defined as the intersection of all relatively  definable subgroups of bounded index, and it turns out to be $\emptyset$-$\bigwedge$-definable and normal. Regarding stable quotients, since stability of hyperdefinable sets is closed under taking products and type-definable subsets (see \cite[Remark 1.4]{MR3796277}), it is clear that there always exists a smallest $A$-$\bigwedge$-definable subgroup $G^{st}_A$ such that the quotient $G/G^{st}_A$ is stable. The main result of \cite{MR3796277} says that under NIP, $G^{st}_A$ does not depend on $A$, and so it is the smallest $\bigwedge$-definable (over parameters) subgroup with stable quotient $G/G^{st}$, and it is in fact $\emptyset$-$\bigwedge$-definable and normal. 
Under NIP, there is also a $\emptyset$-$\bigwedge$-definable subgroup $G^{st,0}$ which is defined as the intersection of all relatively definable (with parameters) subgroups $H$ of $G$ such that $G/H$ is stable. It is interesting to study those  canonical ``components'' as well as quotients by them. To give a non-stable example, consider a monster model $K$ of ACVF, and $G:=(V,+)$, where $V$ is the valuation ring of $K$. Then $G^{st}=G^{st,0}$ is precisely the additive group of the maximal ideal of $V$, and $G/G^{st}$ is the additive group of the residue field.
 
In  \cite{MR3796277}, the authors suggested that it should be true that for groups definable in o-minimal theories, and, more generally, in distal theories (see Definition \ref{definition: distality}), $G^{st}=G^{00}$. This agrees with the intuition that distality should be thought of as something at the opposite pole from stability. As an illustration, consider the unit circle in the monster model of RCF: then $G^{st}=G^{00}$ is the group of infinitesimals and $G^0=G$.  In Section \ref{section: distal theories}, we prove this conjecture in the following more general form (see Corollary \ref{corollary: stable iff bounded}).

\begin{prop}\label{proposition: distal implies bdd}
If $T$ is distal, then every stable hyperdefinable set is bounded.
\end{prop}

This is deduced from the following result (see Theorem \ref{teordistal}). 

\begin{prop}\label{proposition: preservation of distality}
If $T$ is distal, then $T^{heq}$ is distal (by which we mean that all dense indiscernible sequences of hyperimaginaries are distal).
\end{prop}

We prove the above proposition by elaborating on some arguments from \cite{MR3001548}. 

By Hrushovski's theorem (i.e. \cite[Ch. 1, Lemma 6.18]{pillay1996geometric}), we know that a $\bigwedge$-definable group in a stable theory is an intersection of definable groups. 
%
However, although $G/G^{st}$ is stable, it may happen that $G^{st}$ is not an intersection of relatively definable subgroups of $G$,  e.g. in the above example with the unite circle, $G^{st}=G^{00}$ is not an intersection of definable groups. 
In \cite{MR3796277}, the authors stated as a problem to find an example of a definable group $G$ where $G^{00}=G$ but $G^{st} \ne G^{st,0}$ (i.e. $G^{st}$ is not an intersection of definable groups). In Section \ref{section: main example}, we give such an example: it is  the monster model of $\Th((\mathbb{R},+,[0,1])$. It is not clear to us, however, how to find an example of a torsion (equivalently, finite exponent) group $G$ with those properties, or just satisfying $G^{00}\neq G^{st}\neq G^{st,0}$. In Section \ref{section 5}, we make some observations on this problem, 
describing what should be constructed in order to find such an example. Dropping the requirement that $G$ is a torsion group, we give a  large class of examples where $G^0\neq G^{00}\neq G^{st}\neq G^{st,0}$; this does not include an example of finite exponent, as $G$ being of  finite exponent implies that  $G^0 =G^{00}$ by general topological reasons (i.e. compact torsion groups are profinite).

	\section{Characterizations of stability of hyperdefinable sets}\label{section: characterizations of stability}


Let $T$ be a complete, first order theory, and $\C \models T$ a monster model. Let $E$ be a $\emptyset$-$\bigwedge$-definable equivalence relation on a $\emptyset$-$\bigwedge$-definable subset $X$ of $\C^\lambda$ (or a product of sorts), where $\lambda< \kappa$ (from the definition of $\C$). Recall that $E$ is said to be {\em bounded} if $|X/E| < \kappa$.

	In this section, we give some characterizations of stability of the hyperdefinable set $X/E$, analogous to classical characterizations of stability of a first order theory. This involves continuous logic (CL). Assuming NIP, we also give a characterization using generically stable types, which we introduce in the context of $X/E$ (which in fact could be also done in a general CL context). 
	
Since finite models are trivially stable, we will assume that $T$ has infinite models.

It is folklore that $E$ yields a pseudometric (or a set of pseudometrics) on $X$ (see \cite{2010,conant2021separation} and \cite[Chapter 3]{hanson2020thesis} in the language of continuous logic and \cite{10.2178/jsl/1122038916} in the language of CATs), which in turn leads to a presentation of $X/E$ as a type-definable set of imaginaries in the sense of continuous logic.
 Note that in this translation hyperdefinable sets do not translate to continuous logic definable sets. However, for our purposes, it is more convenient to look at the connection with continuous logic in a different way.

We will focus on the first order theory $T$ and treat it as a continuous logic theory, as the aim of this paper is talk about $X/E$ rather than develop continuous logic in general. We will be using some results from \cite{MR2657678} 
but also the formalism from \cite[Subsection 3.1]{hrushovski2021amenability} and \cite[Section 3]{hrushovski2021order}. In particular, by a {\em CL-formula over $A$} we mean a continuous function $\varphi \colon S_n(A)\to \mathbb{R}$.  If $\varphi$ is such a CL-formula, then for any $\bar b\in M^n$ (where $M \models T$) by $\varphi(\bar b)$ we mean $\varphi(\tp(\bar b/A))$; note that the range of every CL-formula is compact. So a CL-formula can be thought of as a function from $\C^n$ to $\mathbb{R}$ which factors through $S_n(A)$ via a continuous map $S_n(A) \to \mathbb{R}$.
 What are called {\em definable predicates}, in finitely many variables and without parameters,  in \cite{MR2657678} are precisely CL-formulas over $\emptyset$, but where the range is contained in $[0,1]$. In any case, a CL-formula can be added as a new CL-predicate and then it becomes a legitimate formula in the sense of continuous logic. It is not so if we allow the domain of a CL-formula to be an infinite Cartesian power of $\C$ (which is necessary to deal with $X/E$ in the case when $\lambda$ is infinite), but still the results from \cite{MR2657678} 
 which we will be using are valid for such generalized continuous logic formulas.

Let $M$ be a model, and $\varphi(x,y)$ a CL-formula over $M$. Let $a\in \C^{|x|}$. 
Then $\tp_{\varphi}(a/M)$ is the function taking $b\in M^{|y|}$ (or $\varphi(x,b)$) to $\varphi(a,b)$, and is called a {\em complete $\varphi(x,y)$-type over $M$}. The space of all complete $\varphi$-types over $M$ is denoted by $S_\varphi(M)$ (it is naturally a quotient of $S(M)$, and the topology on $S_\varphi(M)$ is the quotient topology).
The type $\tp_{\varphi}(a/M)$ is {\em definable} if it is the restriction to $M^{|y|}$ of a CL-formula $\psi(y)$ over $M$, i.e. $\varphi(a,b)=\psi(b)$ for $b \in M^{|y|}$.


	


From now on, let $\mathcal{F}_{X/E}$ be the family of all functions $f : X \times \mathfrak{C}^m \to \mathbb{R}$ which factor through $X/E\times\mathfrak{C}^m$ and can be extended to a CL-formula $\C^{\lambda} \times \C^m \to \mathbb{R}$ over $\emptyset$, where $m$ ranges over $\omega$. 
(Note that, by Tietze extension theorem, a function $f : X \times \C^m \to \mathbb{R}$ extends to a CL-formula over $\emptyset$ iff it factors through the type space $S_{X \times \C^m}(\emptyset)$ via a continuous function $S_{X \times \C^m}(\emptyset) \to \mathbb{R}$.)
For $f\in \mathcal{F}_{X/E}$, a {\em complete $f$-type over $M$} is the function taking  $f(x,b)$ (for $b\in M^{|y|}$) to $f(a,b)$ for some fixed $a \in X$, and is denoted by $\tp_f(a/M)$. We get the space $S_f(M)$ of all complete $f$-types over $M$. A {\em complete $\mathcal{F}_{X/E}$-type over $M$} is 
the union $\bigcup_{f \in \mathcal{F}_{X/E}} \tp_f(a/M)$ for some $a \in X$, and  $S_{\mathcal{F}_{X/E}}(M)$ is the space of all complete  $\mathcal{F}_{X/E}$-types over $M$. The definition of $\tp_f(a/M)$ being definable is the same as in the previous paragraph; a type in $S_{\mathcal{F}_{X/E}}(M)$ is {\em definable} if its restriction to any $f \in \mathcal{F}_{X/E}$ is  definable.

Let $A \subset \C$ (be small). Recall that the complete types over $A$ of elements of $X/E$ can be defined as the $\aut(\C/A)$-orbits on $X/E$,  or the preimages of these orbits under the quotient map, or the partial types defining these preimages, or the classes of the equivalence relation on $S_X(A)$ given in the proof of Remark \ref{2.7}. The space of all such types is denoted by $S_{X/E}(A)$.
	
	\begin{prop}\label{2.3}
		For any $a_1=a'_1/E$, $a_2=a'_2/E$ in $X/E$ and $b_1,b_2\in \mathfrak{C}^m$ 
		$$\tp(a_1,b_1)\neq \tp(a_2,b_2)\iff (\exists f\in \mathcal{F}_{X/E})(f(a'_1,b_1)\neq f(a'_2,b_2)) $$
	\end{prop}

\begin{proof}
Let us define an equivalence relation $E'$ on $X\times \mathfrak{C}^{m}$ by $$ (x_1,y_1) {E'} (x_2,y_2)\iff (x_1/E,y_1)\equiv(x_2/E,y_2). $$
Note that $E'$ is a $\emptyset$-$\bigwedge$-definable, bounded equivalence relation.
	
($\Leftarrow$) Assume $r_1:=f(a_1',b_1)\neq f(a'_2,b_2)=:r_2$ for some $f\in \mathcal{F}_{X/E}$. Since the sets $f^{-1}(r_1)$ and $f^{-1}(r_2)$ are  $\emptyset$-$\bigwedge$-definable and they are unions of ($E \times \{=\}$)-classes, they are unions of $E'$-classes. But they are also disjoint. Hence, $ (a_1',b_1)$ is not $E'$-related to $(a_2',b_2)$, i.e. $\tp(a_1,b_1) \ne \tp(a_2,b_2)$. 

($\Rightarrow$)  Since $E'$ is $\emptyset$-$\bigwedge$-definable and bounded, $\bigslant{(X\times \mathfrak{C}^{m})}{E'}$ 
is a compact  (Hausdorff) topological space (with the {\em logic topology}, in which closed sets are those whose preimages by the quotient map are $\bigwedge$-definable). Since we assume that $\tp(a_1,b_1) \ne \tp(a_2,b_2)$, we have $[(a_1',b_1)]_{E'} \ne [(a_2',b_2)]_{E'}$ in $\bigslant{(X\times \mathfrak{C}^{m})}{E'}$. The space $\bigslant{(X\times \mathfrak{C}^{m})}{E'}$ is $T_{3+\frac{1}{2}}$, so the above two distinct points can be separated by a continuous function $$ h:\bigslant{(X\times \mathfrak{C}^{m})}{E'}\to \mathbb{R} $$ such that $h([(a_1',b_1)]_{E'})=0$ and $h([(a_2',b_2)]_{E'})=1$. Let $\pi_{E'} : X \times \C^m \to  \bigslant{(X\times \mathfrak{C}^{m})}{E'}$ be the quotient map. We conclude that the function $$f:=h\circ \pi_{E'}: X\times \mathfrak{C}^{m} \to \mathbb{R}  $$ satisfies the required conditions.
\end{proof}

We say that $f \in \mathcal{F}_{X/E}$ is {\em stable} if for all $\varepsilon > 0$ there do not exist $a_{i}, b_{i}$ for $i<\omega$ with $a_i \in X$ for each $i$, such that for all $i<j$, $|f(a_{i},b_{j}) - f(a_{j},b_{i})| \geq \varepsilon$ (see \cite[Definition 7.1]{MR2657678} and \cite[Definition 3.8]{hrushovski2021amenability}). By Ramsey's theorem and compactness, $f$ is stable iff whenever $(a_{i},b_{i})_{i<\omega}$ is indiscernible (with $a_i \in X$), then $f(a_{i},b_{j}) = f(a_{j},b_{i})$ for all (some) $i< j$. 
	
	\begin{cor}\label{2.4}
		$X/E$ is stable as a hyperdefinable set if and only if every  $f\in \mathcal{F}_{X/E}$ is stable.
	\end{cor}
	\begin{proof}
		
		
		($\Rightarrow$) Suppose that there is an unstable $f\in \mathcal{F}_{X/E}$. Then there is an indiscernible sequence $(a_i,b_i)_{i<\omega}$ with $a_i \in X$ such that $$f(a_i,b_j)\neq f(a_j,b_i)$$ for all $i<j$.
		Hence, by Proposition \ref{2.3}, $\tp(a_i/E,b_j)\neq \tp(a_j/E,b_i)$ for all $i<j$. Since the sequence $(a_i/E,b_i)_{i<\omega}$ is indiscernible, we conclude that $X/E$ is not stable.
		
		($\Leftarrow$) Suppose that $X/E$ is not stable. Then, there is an indiscernible sequence $(a_i/E,b_i)_{i<\omega}$ with $a_i \in X$ such that $$\tp(a_i/E,b_j)\neq \tp(a_j/E,b_i).$$
		for all $i<j$.
		By Ramsey's theorem and compactness, we can assume that the sequence  $(a_i,b_i)_{i<\omega}$ is indiscernible.
		
		By Proposition \ref{2.3}, we conclude that there is $f\in \mathcal{F}_{X/E}$ such that $f(a_i,b_j)\neq f(a_j,b_i)$ for all $i<j$. Hence, $f$ is not stable.
	\end{proof}
	
The next result follows from \cite[Proposition 7.7]{MR2657678} and its proof. However, one should be a bit careful here. In the case when $\lambda$ is finite and $X=\C^\lambda$, one just applies \cite[Proposition 7.7]{MR2657678}, but in general one should say that the proof of \cite[Proposition 7.7]{MR2657678} goes through working with $f \in \mathcal{F}_{X/E}$ in place of a legitimate continuous logic formula $\varphi$. Also, since we are working in the first order theory $T$ treated as a continuous logic theory, models are discrete spaces and the density characters 
of models are just cardinalities. 
The density character of $S_f(M)$ (denoted by $||S_f(M)||$) is computed with respect to a certain metric on $S_f(M)$ defined after Definition 6.1 in \cite{MR2657678}.

\begin{fact}[\cite{MR2657678}, Proposition 7.7]\label{equiv1}
	Let $f\in \mathcal{F}_{X/E}$. The following conditions are equivalent:
	\begin{enumerate}
		\item $f$ is stable.	
		\item For every $M\models T$, every $p\in S_f(M)$ is definable.
		\item For every $M\models T$, $||S_f(M)||\leq | M |$.
		\item For every $M\models T$, $|S_f(M)|\leq | M |^{\aleph_0}$.
		\item There is $\mu \geq |T|$ such that when $M\models T$ and $|M| \leq \mu$, then $||S_f(M)|| \leq \mu$.
		\item For every $\mu =\mu^{\aleph_0} \geq |T|$, when $M\models T$ and $|M| \leq \mu$, then $|S_f(M)| \leq \mu$.
	\end{enumerate}
\end{fact}


\begin{cor}\label{corollary: 2.4}
	The following conditions are equivalent:
	\begin{enumerate}
		\item $\forall f \in \mathcal{F}_{X/E}$ $f$ is stable.
		\item $\forall M\models T$ $\forall f\in \mathcal{F}_{X/E}$ $\forall p\in S_f(M)$ $p$ is definable.
		\item $\exists \mu\geq \lvert T \rvert$ s.t. $\forall M\models T$ if  $M\models T$ and $|M| \leq \mu$, then $| S_{\mathcal{F}_{X/E}}(M)| \leq \mu$.
		\item $\forall \mu = \mu ^{|T|+\lambda} \geq \lvert T \rvert$ $\forall M\models T$ if  $M\models T$ and $|M| \leq \mu$, then $| S_{\mathcal{F}_{X/E}}(M)| \leq \mu$.
	\end{enumerate}
\end{cor}

\begin{proof}
This follows easily from Fact \ref{equiv1}. Only (1) $\Rightarrow$ (4) is a bit more delicate, which we will explain. So assume (1). Then we have (6) from Fact \ref{equiv1}. 

Fix $m < \omega$. By Stone-Weierstrass theorem, the first order formulas restricted to $X \times \C^m$ generate a dense subalgebra $\mathcal{A}_m$ of cardinality at most $|T|+ \lambda$ of the Banach algebra $\mathcal{B}_m$ of all functions $f: X \times \C^m \to \R$ which extend to a CL-formula from $\C^\lambda \times \C^m$ to $\R$. As the family $\mathcal{F}^m_{X/E}$ of those functions from $\mathcal{B}_m$ which factor through  $X/E \times \C^m$ is a subspace of $\mathcal{B}_m$, it also has a dense subset $\mathcal{D}_m$ of cardinality at most $|T|+\lambda$. Since clearly $\mathcal{F}_{X/E} = \bigcup_{m<\omega} \mathcal{F}^m_{X/E}$, we get that the complete $\mathcal{F}_{X/E}$-type over $M$ of an element $a \in \C^\lambda$ is determined by $\bigcup_{m<\omega} \bigcup_{f \in \mathcal{D}_m} \tp_f(a/M)$. Using this and (6) from Fact \ref{equiv1}, one easily gets (4) in Corollary \ref{corollary: 2.4}.
\end{proof}

	
	\begin{remark}\label{2.7}
		For any model $M$ of $T$ there is a natural bijection $$ S_{X/E}(M)\to S_{\mathcal{F}_{X/E}}(M). $$
	\end{remark}
\begin{proof}
$S_{\mathcal{F}_{X/E}}(M)$ can be seen as $\bigslant{S_X(M)}{\sim_{\mathcal{F}_{X/E}}}$,
where for every $p,q \in S_X(M)$ and some (equivalently, any) $a_1' \models p$ and $a_2' \models q$:
$$p\sim_{\mathcal{F}_{X/E}} q \iff (\forall f(x,y)\in \mathcal{F}_{X/E}) (\forall b\in M^{|y|})( f(a_1',b)=f(a_2',b)).$$
On the other hand, $S_{X/E}(M)=\bigslant{S_X(M)}{\sim_E}$, where for every $p,q \in S_X(M)$ and some (equivalently, any) $a_1' \models p$ and $a_2' \models q$:
$$p\sim_Eq\iff a_1'/E\equiv_M a_2'/E \iff (\forall m<\omega)(\forall b\in M^m)( (a_1'/E,b)\equiv (a_2'/E,b) ).$$
By Proposition \ref{2.3}, $p\sim_{\mathcal{F}_{X/E}} q$ iff $p\sim_Eq$. Hence, the conclusion follows.
\end{proof}

From the previous results,
we get some characterizations of stability of $X/E$.
	\begin{cor}\label{2.8}
		The following conditions are equivalent:\begin{enumerate}
			\item $X/E$ is stable.
			\item $\forall f\in \mathcal{F}_{X/E}$ $(f$ is stable$)$.
			\item $\forall$ $M\models T$ $\forall f\in \mathcal{F}_{X/E}$ $\forall p\in S_f(M)$ $(p$ is definable$)$.
			\item $\exists \mu\geq \lvert T\rvert$ $\forall M\models T$ $(\lvert M\rvert\leq \mu \implies \lvert S_{\mathcal{F}_{X/E}}(M)\rvert\leq \mu)$.  
			\item $\exists \mu\geq \lvert T\rvert$ $\forall M\models T$ $(\lvert M\rvert\leq \mu \implies \lvert S_{X/E}(M)\rvert\leq \mu)$. 
			\item $\forall \mu = \mu ^{|T|+\lambda} \geq \lvert T \rvert$ $\forall M\models T$ $(\lvert M\rvert\leq \mu \implies \lvert S_{X/E}(M)\rvert\leq \mu)$.
		\end{enumerate}
	\end{cor}

	\begin{proof}
The equivalence between (1), (2), (3), and (4) follows from Corollaries \ref{2.4} and \ref{corollary: 2.4}. The equivalence of (4) and (5) follows from Remark \ref{2.7}. The equivalence of (2) and (6) follows from Corollary \ref{corollary: 2.4} and Remark \ref{2.7}.
	\end{proof}

As an application of the characterization from Corollary \ref{2.8}(6), we give a quick proof of Remark 2.5(iii) from \cite{MR3796277} that $G^{st}$ does not have proper hyperdefinable, stable quotients (which was left to the reader in  \cite{MR3796277}). 
Namely, suppose $H < G^{st}$ is a proper $A$-$\bigwedge$-definable subgroup for some $A$; add all elements of $A$ as new constants.
We need to show that $G^{st}/H$ is unstable. By minimality of $G^{st}$, $G/H$ is unstable. So, by Corollary \ref{2.8}(6), there is $\mu = \mu ^{|T|+\lambda} \geq \lvert T \rvert$, a model $M$ of $T$ of cardinality $\mu$, and a sequence $(g_i)_{i<\mu^+}$ in $G$ such that $\tp(g_iH/M) \ne \tp(g_jH/M)$ for all $i \ne j$. Since $G/G^{st}$ is stable, by Corollary \ref{2.8}(6), there is a subset $I$ of $\mu^+$ of cardinality $\mu^+$ such that $g_iG^{st}\equiv_M g_jG^{st}$ for all $i,j \in I$. Fix $i_0 \in I$ and put $I_0:=I \setminus \{i_0\}$. Mapping all $g_i$, $i \in I_0$, by automorphisms over $M$, we can assume that they are all in the coset $g_{i_0}G^{st}$. Then $g_i':= g_{i_0}^{-1}g_{i} \in G^{st}$ for all $i \in I_0$. Moreover, take any $N \succ M$ containing $g_{i_0}$ and with $|N|= \mu$. Then $\tp(g_i'H/N) \ne \tp(g_j'H/N)$ for every distinct $i,j \in I_0$. Hence,  by Corollary \ref{2.8}, $G^{st}/H$ is unstable.

Next, we recall the definition of NIP for a hyperdefinable set, given in \cite[Remark 2.3]{MR3796277}, and we introduce the notion of generic stability 
for hyperimaginary types.
	\begin{defin}
	A hyperdefinable set $X/E$ has {\em NIP} if there	do not exist an indiscernible sequence $(b_i)_{i<\omega}$ and $d\in X/E$ such that
	$((d, b_{2i}, b_{2i+1}))_{i<\omega}$ is indiscernible and $tp(d, b_0) \neq tp(d, b_1)$. $($Note that
	the $b_i$ can be anywhere, not necessarily in $X/E$.$)$
\end{defin}

Let $p \in S_{X/E}(\C)$ be invariant over $A$. A {\em Morley sequence in $p$ over $A$} is a sequence $(a_i)_i$ of elements of $X/E$ such that $a_i \models p |_{Aa_{<i}}$. As in the home sort, by a standard argument, one can check that Morley sequences (of a given length) in $p$ over $A$  are $A$-indiscernible and have the same type over $A$.

		\begin{defin}
		An $A$-invariant type $p \in S_{X/E}(\C)$ is {\em generically stable} if every  Morley sequence $(a_i/E)_{i<\omega+\omega}$ in $p$ over $A$ satisfies $(\forall \varepsilon >0)$ $(\forall r\in \mathbb{R})$ $(\forall s\leq r-\varepsilon)$ $(\forall f(x,y)\in \mathcal{F}_{X/E})$ $(\forall b\in \mathfrak{C}^{\lvert y \rvert})$
\useshortskip
\begin{equation*} 
\begin{gathered}
\{i<\omega+\omega: f(a_i,b)\leq s\} \text{ is finite}\\
\text{or}\\
\{i<\omega+\omega: f(a_i,b)\geq r\}\text{ is finite.}
\end{gathered}
\end{equation*}
	\end{defin}

Generic stability of $p$ does not depend on the choice of $A$ over which $p$ is invariant. Using compactness theorem, one can show the following characterization.

	\begin{prop}\label{proposition: characterization of gen. stab.}
	An $A$-invariant type  $p\in S_{X/E}(\C)$ is generically stable if and only if for every $\varepsilon>0$ and $f(x,y)\in\mathcal{F}_{X/E}$ there exists $N(f,\varepsilon)\in \mathbb{N}$ for which there is no Morley sequence $(a_i/E)_{i<\omega}$ in $p$ over $A$, subsequences $R,S$ each of which of length at least $N(f,\varepsilon)$, and $b\in \C^{|y|}$ such that  $\lvert f(a_i,b)-f(a_j,b)\rvert \geq \varepsilon$ for all $a_i/E\in R$ and $a_j/E\in S$.
\end{prop}

Our next goal is to extend Corollary \ref{2.8} to:

	\begin{teor}\label{equivteor}
	Assume $X/E$ has NIP. The following conditions are equivalent:
	\begin{enumerate}
		\item $X/E$ is stable.
		\item $\forall$ $M\models T$ $\forall f\in \mathcal{F}_{X/E}$ $\forall p\in S_f(M)$ $($$p$ is definable$)$.
		\item $\exists \lambda\geq \lvert T\rvert$ $\forall M\models T$ $(\lvert M\rvert\leq \lambda \implies \lvert S_{X/E}(M)\rvert\leq \lambda)$.
		\item Any indiscernible sequence of elements of $X/E$ is totally indiscernible.
		\item Any global invariant $($over some $A$$)$ type $p \in S_{X/E}(\C)$ is generically stable.
	\end{enumerate}
\end{teor}
From the proof of this theorem, it will be clear that $(1)$, $(2)$, $(3)$, and $(5)$ are equivalent and imply $(4)$ without the NIP assumption; NIP is used to prove the implication from $(4)$ to $(1)$.

In order to prove Theorem \ref{equivteor}, we will first prove some results about hyperdefinable sets with NIP and about generically stable types.

From now on, EM will stand for Erenfeucht-Mostowski.  By the {\em EM-type of a sequence $I=(a_i/E)_{i \in \mathcal{I}}$} (symbolically, $\EM(I)$) we mean the set of all formulas $\varphi(x_1,\dots,x_n)$ such that for every $i_1 < \dots<i_n \in \mathcal{I}$, $\varphi(a_{i_1}',\dots,a_{i_n}')$ holds for all $a_{i_1}' \in [a_{i_1}]_E, \dots, a_{i_n}' \in [a_{i_n}]_E$, where $n$ ranges over $\omega$. We say that an indiscernible sequence $J$ {\em satisfies the EM-type of $I$} if $\EM(I) \subseteq \EM(J)$. By Ramsey's theorem and compactness, for every sequence $I$ there is an indiscernible sequence $J$ satisfying $\EM(I)$; we can even require that $J$ has an indiscernible sequence of representatives.

In the next three lemmas, $X/E$ and $Y/F$ are arbitrary $\emptyset$-hyperdefinable sets.

\begin{lema}\label{lemma: 2.11} The following conditions are equivalent:
\begin{enumerate}
\item       There	exists an indiscernible sequence $(b_i)_{i<\omega}$ in $Y/F$ and $d\in X/E$ such that
	$((d, b_{2i}, b_{2i+1}))_{i<\omega}$ is indiscernible and $\tp(d, b_0) \neq \tp(d, b_1)$.

\item 	There exists an  indiscernible sequence $(b_i)_{i<\omega}$ in $Y/F$ and $d\in X/E$ such that 
	\useshortskip
	\begin{align*}
		\tp(d,b_i)=\tp(d,b_0)&\iff i \text{ even}\\
		\tp(d,b_i)=\tp(d,b_1)&\iff i \text{ odd}.
	\end{align*}
\end{enumerate}
\end{lema}

\begin{proof}
	$(1) \Rightarrow (2)$ is clear.
	
	$(2) \Rightarrow (1)$ Assume (2). 
Then $\tp(d,b_0)\neq \tp(d,b_1)$.
As $(b_i)_{i<\omega}$ is indiscernible, so is $(b_{2i},b_{2i+1})_{i<\omega}$. Choose  an indiscernible sequence $(\tilde{d},b_{2i},b_{2i+1} )_{i<\omega}$ satisfying the EM-type of $(d,b_{2i},b_{2i+1})_{i<\omega}$. Then, $\tilde{d}\in X/E$ and $(b_i)_{i<\omega}$ witness (1).
\end{proof}

	\begin{lema}
		Let $p, q \in S_{X/E \times Y/F}(\emptyset)$ be distinct types. 
		Then the following conditions are equivalent:
		\begin{enumerate}
			\item There exists an indiscernible sequence $(b_i)_{i<\omega}$ in $Y/F$ and an element $d\in X/E$ with: 
			\useshortskip
			\begin{align*}
				\tp(d,b_i)=p&\iff i \text{ even}\\
				\tp(d,b_i)=q&\iff i \text{ odd}.
			\end{align*}
			\item There is a sequence $(b_i)_{i<\omega}$ in $Y/F$ (not necessarily
			indiscernible) which is shattered by $(p,q)$ in the sense that for every
			$I \subseteq \omega$ there is $d_I \in X/E$ with 
			\useshortskip
			\begin{align*} \tp(d_I,b_i) =p \iff i
			\in I\\ \tp(d_I,b_i)=q \iff i \notin I.
		\end{align*}
		\end{enumerate}
	\end{lema}

	\begin{proof}
		$(1)\Rightarrow (2)$ Let $(b_i)_{i<\omega}$ in $Y/F$ and $d\in X/E$ witness $(1)$. Let $I\subseteq \omega$. We can find an increasing one to one map $\tau:\omega \to \omega$ such that for all $i\in \omega$, $\tau(i)$ is even iff $i\in I$. By indiscernibility, the map sending $b_i$ to $b_{\tau(i)}$ for all $i\in \omega$ can be extended to an automorphism $\sigma$. The element $d_I:=\sigma^{-1}(d)$ satisfies the conditions in $(2)$.
		
		$(2)\Rightarrow (1)$ Let $(b_i)_{i<\omega}$ witness $(2)$. We can find an indiscernible sequence $(c_i)_{i<\omega}$ in $Y/F$ satisfying the EM-type of $(b_i)_{i<\omega}$. It follows that for any two disjoint finite sets $I_0,I_1\subseteq \omega$, the partial type 
		$$\{ p(x;c_i): i\in I_0\}\cup \{q(x;c_i): i\in I_1\}$$ 
is consistent. By compactness, the sequence $(c_i)_{i<\omega}$ is shattered by $(p,q)$. In particular, there is $d\in X/E$ such that $\tp(d,c_i)=p$ iff $i$ is even and $\tp(d,c_i)=q$ iff $i$ is odd.
	\end{proof}

\begin{lema}
Let $p,q \in S_{X/E \times Y/F}(\emptyset)$ be distinct. Then there exists an infinite sequence in $Y/F$ shattered by $(p,q)$ if and only if there exists an infinite sequence in $X/E$ shattered by $(p^{opp},q^{opp})$, where $p^{opp}(x,y):=p(y,x)$.
\end{lema}
\begin{proof}
	 Let $(b_i)_{i<\omega}$ be a sequence in $Y/F$ shattered by $(p,q)$. By compactness, we can find a sequence $(c_i)_{i\in \mathcal{P}(\omega)}$ in $Y/F$ which is shattered by $(p,q)$ as witnessed by the family $\{ d_I: I\subseteq \mathcal{P}(\omega) \} \subseteq X/E$. Consider the sequence $(d_j)_{j< \omega}$ in $X/E$, where $d_j:=d_{I_j}$ and $I_j:=\{X\subseteq \omega: j\in X\}$. Then for any $J\subseteq \omega$ we have: 
	 \begin{align*} \tp(d_j,c_J) =p \iff j
	 	\in J\\ \tp(d_j,c_J)=q \iff j \notin J,
	 \end{align*}
 because $j\in J$ iff $J\in I_j$. 
 Thus, $(d_j)_{j<\omega}$ is shattered by $(p^{opp},q^{opp})$.
 
 The converse follows by symmetry.
\end{proof}

From the last three lemmas, one easily deduces the following

\begin{cor}\label{corollary: switched role of X/E and anything}
	$X/E$ has NIP if and only if there do not exist an indiscernible sequence $(b_i)_{i<\omega}$ of elements of $X/E$ and $d$ (from anywhere) such that the sequence $(d,b_{2i},b_{2i+1})_{i<\omega}$ is indiscernible and $\tp(d,b_0)\neq \tp(d,b_1)$.
\end{cor}

The next lemma is analogous to the finite-cofinite lemma on NIP theories.

\begin{lema}\label{fincofin}
	 Suppose that $X/E$ has NIP. Let $(a_i)_{i\in I}$ be an infinite, totally indiscernible sequence of elements of $X/E$ and $b$ any tuple from $\C$. Then, for any $j_0,j_1\in I$, whenever $\tp(a_{j_0},b)\neq \tp(a_{j_1},b)$, either

\begin{equation*} 
\begin{gathered}
I_0:=\{i\in I: \tp(a_i,b)=\tp(a_{j_0},b)\} \text{ is finite}\\
\text{or}\\
I_1:=\{i\in I:\tp(a_i,b)=\tp(a_{j_1},b)\} \text{ is finite.}
\end{gathered}
\end{equation*}
\end{lema}
\begin{proof}
	Otherwise, we can build a sequence $(i_k)_{k<\omega}$ of pairwise distinct elements of $I$ so that $i_k\in I_0$ iff $k$ is even, and $i_k\in I_1$ iff $k$ is odd. Then, the sequence $(a_{i_k})_{k<\omega}$ is indiscernible and $$\tp(a_{i_k},b)=\tp(a_{j_0},b)\iff k \text{ even}$$ $$\tp(a_{i_k},b)=\tp(a_{j_1},b)\iff k \text{ odd},$$
	which by Lemma \ref{lemma: 2.11} and Corollary \ref{corollary: switched role of X/E and anything} imply that $X/E$ does not have NIP, a contradiction. 
\end{proof}

	\begin{defin}
		The median value connective $\med_n:[0,1]^{2n-1}\to[0,1]$ is defined by $$ \med_n(t_{<2n-1})=\max_{\substack{w\subseteq 2n-1 \\ \lvert w\rvert=n}}\min_{i\in w} t_i= \min_{\substack{w\subseteq 2n-1 \\ \lvert w\rvert=n}}\max_{i\in w} t_i. $$
	\end{defin}

	This connective, as its name indicates, literally computes the median value of the list of arguments.
	
For  $f(x,y) \in \mathcal{F}_{X/E}$ and $N \in \mathbb{N}^+$, put  
$$d^N f(y,x_{<2N-1}):=\med_{N}(f(x_i,y): i<2N-1).$$	

	As in classical model theory, generically stable types are definable, which follows from the next proposition. For  $p \in S_{X/E}(\C)$, $f(x,y) \in \mathcal{F}_{X/E}$, and $b \in \C^{|y|}$, by $f(x,b)^p$ we mean the value of $p$ at $f(x,b)$ for $p$ treated as an element of $S_{\mathcal{F}_{X/E}}(\C)$ as explained in Remark \ref{2.7} (in other words, it is $f(a,b)$ for $a/E \models p$). We say that $p$ is {\em definable} if it is so as a type in $S_{\mathcal{F}_{X/E}}(\C)$.

	\begin{prop}\label{2.21}
		If $p \in S_{X/E}(\C)$ is generically stable over $A$, then for any $f(x,y)\in \mathcal{F}_{X/E}$, $b\in \mathfrak{C}^{\lvert y\rvert}$, and $(a_i/E)_{i<\omega}$ a Morley sequence in $p$ over $A$, $$ \lvert f(x,b)^p- d^{N(f,\varepsilon)} f(b,a_{<2N(f,\varepsilon)-1})\rvert \leq \varepsilon,$$
		where $N(f,\varepsilon)$ is a number as in Proposition \ref{proposition: characterization of gen. stab.}. 
	\end{prop}
	
	\begin{proof}
		Suppose this is not true. Write $N$ for $N(f,\varepsilon)$. Then, we have two cases:

		1) $d^N f(b,a_{<2N-1})-\varepsilon>f(x,b)^p$. This implies $$\max_{\substack{w\subseteq 2N -1 \\ \lvert w\rvert=N}}\min_{i\in w} f(a_i,b)>f(x,b)^p+\varepsilon.$$
		Hence, there is $w$ of size $N$ such that for all $i\in w$ we have $f(a_i,b)>f(x,b)^p + \varepsilon$. Taking a Morley sequence in $p$ over $Ab(a_i)_{i\in w}$, we get a contradiction with the choice of $N(f,\varepsilon)$.
		
		2)	$d^N f(b,a_{<2N-1})+\varepsilon<f(x,b)^p$. This case is analogous to the previous, using the other definition of the median value connective.
	\end{proof}
	
\begin{cor}\label{corollary: gen stab implies definable}
All generically stable types in $S_{X/E}(\C)$ are definable.
\end{cor}

\begin{proof}
Consider any generically stable (over $A$) $p \in S_{X/E}(\C)$ and $f(x,y)\in \mathcal{F}_{X/E}$. Let $(a_i/E)_{i<\omega}$ be  a Morley sequence in $p$ over $A$.
Define a CL-formula $df(y,z)$ to be the forced limit of the sequence $(d^{N(f,2^{-n})} f(y,x_{<2N(f,2^{-n})-1}))_{n<\omega}$ (see Definitions 3.6 and 3.8 in \cite{MR2657678}). By the last proposition and \cite[Lemma 3.7]{MR2657678}, we get that $df(y,(a_{i})_{i<\omega})$ 
 is the $f$-definition of $p$
\end{proof}

Let $M \prec \C$ (small), $f \in \mathcal{F}_{X/E}$, $p \in S_f(M)$, and $q \in S_X(\C)$. It is clear what it means that $q$ extends $p$, namely: for $d \models q$ and for all $b \in M^{|y|}$, $f(x,b)^p = f(d,b)$ (where $f$ is canonically extended to a bigger monster model to which $d$ belongs). 
This is equivalent to saying that the partial type over $M$ defining $\{d \in \C^{|x|}: f(d,b) =f(x,b)^p \text{ for all } b \in M^{|y|}\}$ is contained in $q$. 
Thus, using the well-known fact that each partial type over $M$ (even in infinitely many variables) extends to a global coheir over $M$ (i.e. global type finitely satisfiable in $M$), we have:

\begin{fact}\label{fact: instead of BY}
Every type $p \in S_f(M)$ has an extension to a global type $q \in S_X(\C)$ finitely satisfiable in $M$.
\end{fact}

Finally, we present the proof of the main result of this section.

\begin{proof}[Proof of Theorem \ref{equivteor}] $(1)\Leftrightarrow (2)\Leftrightarrow (3)$ is a part of Corollary \ref{2.8}. 
	
	$(1)\Rightarrow (4)$ Suppose that the sequence $(a_i)_{i<\omega}$ in $X/E$ is a counter-example. Without loss of generality we can replace $\omega$ by $\mathbb{Q}$. Then, there are rational numbers $i_0<\dots <i_{n-1}$ and a natural number $j<n-1$ such that $$\tp(a_{i_j},a_{i_{j+1}}/A)\neq \tp(a_{i_{j+1}},a_{i_j}/A),$$ where $A$ is the set of all $a_{i_k}$ for $k<n$ such that $j\neq k \neq j+1$. Choose any rationals $l_0<l_1<\dots$ in the interval $(i_j,i_{j+1})$. Let $b_i=a_{l_i}$ for $i<\omega$. Then, the sequence $(b_i)_{i<\omega}$ is $A$-indiscernible and $\tp(b_i,b_j/A)\neq \tp(b_j,b_i/A)$ for all $i<j<\omega$. This contradicts the stability of $X/E$.
	
	$(4)\Rightarrow (1)$ Assume that $X/E$ is unstable and $(4)$ holds. Since $X/E$ is unstable, there exists an indiscernible sequence $(a_i/E,b_i)_{i\in\mathbb{Z}}$ such that $\tp(a_i/E,b_j)\neq \tp(a_j/E,b_i)$ for $i<j\in \mathbb{Z}$. The sequence $(a_i/E)_{i\in\mathbb{Z}}$ is totally indiscernible and the sets $\{i\in\mathbb{Z}: \tp(a_i/E,b_0)=\tp(a_1/E,b_0)\}$ and $\{i\in\mathbb{Z}: \tp(a_i/E,b_0)=\tp(a_{-1}/E,b_0)\}$ are both infinite, contradicting Lemma \ref{fincofin} and the assumption that $X/E$ has NIP.
	
	$(1)\Rightarrow (5)$ Let $p\in S_{X/E}(\C)$ be invariant (over some A) but not generically stable. Then, there exists a Morley sequence $(a_i/E)_{i\in \omega + \omega}$ in $p$ over $A$, $\varepsilon>0$, $r\in \mathbb{R}$, $s\leq r-\varepsilon$, $f(x,y)\in \mathcal{F}_{X/E}$, and $b \in \C^{|y|}$ such that the sets $ \{i\in \omega + \omega: f(a_i,b)\leq s\} $ and $ \{i\in \omega + \omega: f(a_i,b)\geq r\}$ are both infinite. There are two possible cases: either there are infinitely many alternations, or after removing a finite number of elements $a_i$: 
	$$ (f(a_i,b)\leq s \iff i<\omega) \; \text{ or } \; (f(a_i,b)\geq r \iff i<\omega).$$
By the indiscernibility of $(a_i/E)_{i \in \omega +\omega}$, Ramsey's theorem and compactness, in each case we get a contradiction with the stability of $f$.
	
	$(5)\Rightarrow (2)$ Consider any $f\in \mathcal{F}_{X/E}$ and $p\in S_{f}(M)$. By Fact \ref{fact: instead of BY}, choose $q'\in S_{X}(\mathfrak{C})$ extending $p$ which is a coheir over $M$; let $q \in S_{X/E}(\C)$ be induced by $q'$. Being a coheir over $M$, $q'$ is $M$-invariant; so $q$ is $M$-invariant, hence generically stable by (5). By Corollary \ref{corollary: gen stab implies definable}, $q$ is definable. Denote the $f$-definition of $q$ by $\psi$. Since $q$ is $M$-invariant, so is the CL-formula $\psi$. Therefore, $\psi$ is definable over $M$ (i.e. a CL-formula over $M$). Hence, $p$ is definable by the CL-formula $\psi$.
	\end{proof}

	\section{Distal hyperimaginary sequences}\label{section: distal theories}

The goal of this section is to prove Proposition \ref{proposition: preservation of distality} and deduce Proposition \ref{proposition: distal implies bdd}, which in turn confirms the prediction from \cite{MR3796277} that in a distal theory $G^{st}=G^{00}$.


We work in a monster model $\C$ of a complete, first order theory $T$ with NIP.	 This section is based on \cite{MR3001548}, in particular the next two definitions are from there.

\begin{defin}
    For any indiscernible sequence $I$, if $I = I_1 + I_2$ (the concatenation of $I_1$ and $I_2$), we say that $\mathfrak{c}=(I_1, I_2)$ is a {\em cut} of $I$.
    
    We write $(I_1',I_2') \unlhd (I_1,I_2)$ if $I_1'$ is an end segment of $I_1$ and $I_2'$ an initial
segment of $I_2$.
    
    If $J \subset I$ is a convex subsequence, a cut $\mathfrak{c} = (I_1, I_2)$ is said to be {\em interior to $J$} if $I_1 \cap J$ and $I_2 \cap  J$ are infinite.

    A cut is {\em Dedekind} if both $I_1$ and $I^*_2$  ($I_2$ with the reversed order) have infinite cofinality.
    
    A {\em polarized cut} is a pair $(\mathfrak{c},\varepsilon)$, where $\mathfrak{c}$ is a cut $(I_1,I_2)$ and
$\varepsilon \in \{1,2\}$ is such that $I_\varepsilon$ is infinite.

    If $\mathfrak{c} = (I_1, I_2)$ is a cut, we say that a tuple $b$ {\em fills} $\mathfrak{c}$ if $I_1 + b + I_2$ is indiscernible.
\end{defin}

Sometimes, if it is clear that the tuple $b$ fills some cut $\mathfrak{c}=(I_1,I_2)$ of $I$, we will write $I\cup \{b\}$ instead of $I_1+b+I_2$. And similarly, in the case of two elements $a,b$ filling respectively distinct cuts $\mathfrak{c}_1, \mathfrak{c}_2$, abusing notation, we will  write $I \cup \{a\} \cup \{b\}$ for the associated concatenation.

\begin{defin}\label{definition: distality}
    A dense indiscernible sequence $I$ is {\em distal} if for any distinct Dedekind cuts $\mathfrak{c}_1,\mathfrak{c}_2$, if $a$ fills $\mathfrak{c}_1$ and $b$ fills $\mathfrak{c}_2$, then $I \cup\{a\} \cup \{b\}$ is indiscernible.
    
    The theory $T$ is {\em distal} if all dense indiscernible sequences (of tuples from the home sort) are distal.
 
    We say that {\em $T^{heq}$ is distal} if all dense indiscernible sequences $(a_i/E)_{i \in \mathcal{I}}$ of hyperimaginaries (where $E$ is $\emptyset$-$\bigwedge$-definable) are distal. 
\end{defin}
	
Let $E$ be a $\emptyset$-$\bigwedge$-definable equivalence relation on $\C^\lambda$, and let $\pi_E: \C^\lambda \to \C^\lambda /E$ be the quotient map. 

The next lemma is a variant of \cite[Lemma 2.8]{MR3001548} for hyperimaginaries.

	\begin{lema}\label{2.8'}
		Let $I=(a_i/E)_{i\in \mathcal{I}}$ be a dense indiscernible sequence and $A\subset \C$ a (small) set of parameters. Let $(\mathfrak{c}_i)_{i<\alpha}$ be a sequence of pairwise distinct Dedekind cuts in $I$. For each $i<\alpha$ let $d_i$ fill the cut $\mathfrak{c_i}$. Fix a polarization of each $\mathfrak{c_i}$, $i < \alpha$. Then there are $(d_i')_{i<\alpha}$ satisfying $(d_i')_{i<\alpha}\equiv_I (d_i)_{i<\alpha}$  such that for every formula  $\theta$ with parameters from $A$ and $i<\alpha$: 
		if $$\pi_E^{-1}(d_i')\subseteq 
		\theta(\mathfrak{C}),$$
		then $$ \pi_E^{-1}(a_j/E)\not\subseteq \neg \theta(\mathfrak{C}) $$
		for $a_j/E$ from a co-final fragment of  the left part of $\mathfrak{c}_i$ if $\mathfrak{c}_i$ is left-polarized, or from a co-initial fragment of the right part of $\mathfrak{c}_i$ if $\mathfrak{c}_i$ is right-polarized.
	\end{lema}
	\begin{proof}

To simplify notation, let all the cuts $\mathfrak{c}_i$ be left-polarized. 
The negation of the conclusion says that for every $(d_i')_{i<\alpha} \equiv_{I} (d_i)$ there exists  $i< \alpha$ and a formula $\theta(x)$ over $A$ such that $$\pi_E^{-1}(d_i') \subseteq \theta(\mathfrak{C})$$ and $$ \pi_E^{-1}(a_j/E)\subseteq \neg \theta(\mathfrak{C}) $$ for all $a_j/E$ in some end segment of $\mathfrak{c}_i$. For $i<\alpha$ put $C_i(x_i):=\{\varphi(x_i) \in L(A) :  \pi_E^{-1}(a_j/E)\subseteq \varphi(\C) \text{ for all $a_j/E$ in some end segment of } \mathfrak{c}_i  \}$. Note that these sets are closed under conjunction. 
The negation of the conclusion is equivalent to
$$(*) \;\;\;\;\;\;\;\;\; \tp((d_i)_{i<\alpha})/I)\cup \bigcup_{i<\alpha}C_i(x_i) \text{ is inconsistent,}$$ 
which in turn is equivalent to the existence of a finite subset $J$ of $\alpha$ such that $\tp((d_i)_{i\in J})/I)\cup \bigcup_{i \in J}C_i(x_i)$ inconsistent. Therefore, without loss of generality, $\alpha<\omega$. We will show a detailed proof for $\alpha=2$; the proof for an arbitrary $\alpha \in \omega$ is the same.

Suppose the conclusion fails. By $(*)$, choose a finite subsequence $I_0$ of $I$ so that $\tp(d_0,d_1/I_0)\cup C_0(x_0)\cup C_1(x_1)$ is inconsistent, and let $\varphi_0(x_0)\in C_0(x_0)$ and  $\varphi_1(x_1) \in C_1(x_1)$ be formulas witnessing it.  Now, for $i \in \{0,1\}$ take $(J_i,J_i')\unlhd \mathfrak{c}_i$ such that $\pi_E^{-1}(a_j/E)\subseteq \varphi_i(\mathfrak{C})$ for all $a_j/E\in J_i$, $J_i\cup J_i'$ contains no element of $I_0$, and $(J_0 \cup J_0') \cap (J_1 \cup J_1') =\emptyset$.
		
		\begin{claim}
			For every two cuts $\mathfrak{d}_0,\mathfrak{d}_1$ in $I$ respectively interior to $J_0,J_1$ we can find hyperimaginaries $e_0$ filling $\mathfrak{d}_0$ and $e_1$ filling $\mathfrak{d}_1$ such that $\tp(e_0,e_1/I_0)=\tp(d_0,d_1/I_0)$.
		\end{claim}
		
\begin{proof}[Proof of claim]
			Consider any finite  $K\subseteq I$. For $i \in \{0,1\}$, the cut $\mathfrak{d}_i$ decomposes $K$ into $L_i^- +L_i^+$. It is enough to find $e_0,e_1$ such that \begin{align*}
				\tp(L_1^-,e_0,L_1^+)&\subset \EM(I),\\
				\tp(L_2^-,e_1,L_2^+)&\subset \EM(I),\\
				\tp(e_0,e_1/I_0)&=\tp(d_0,d_1/I_0),
			\end{align*}
where $\EM(I)$ denotes the Erenfeucht-Mostowski type of $I$.
			
			We can decompose $K$ into sequences $K_0\subseteq J_0+J_0'$, $K_1\subseteq J_1+J_1'$, and $K_2\subseteq I \setminus ((J_0+J_0') \cup (J_1+J_1'))$. Next, we construct new finite sequences $K_0'$, $K_1'$, $K_2'$, and $K'$ in the following way: $K_2'=K_2$; for every element $a\in K_0$ we take $a'\in J_0 \cup J_0'$ such that $a'$ is in the same relative position to $\mathfrak{c}_0$ as $a$ was to $\mathfrak{d}_0$ and also preserving the order between elements, and we define $K_0'$ to be the constructed sequence of the $a'$'s; for $K_1'$ we proceed in an analogous manner; finally, $K':=K_0'\cup K_1'\cup K_2'$ written as a sequence in an obvious order provided by the construction. By the indiscernibility of the sequence $I$, there is $\sigma\in \aut(\C)$ such that $\sigma(KI_0)=K'I_0$. The elements $e_0:=\sigma^{-1}(d_0)$ and $e_1:=\sigma^{-1}(d_1)$ satisfy the desired conditions.
		\end{proof}

		Fix $\mathfrak{d}_0,\mathfrak{d}_1$ as in the claim, and choose $e_0=:e^0_0$ and $e_1=:e^0_1$ provided by the claim. By the choice of $\varphi_i(x)$, $\pi_E^{-1}(e^0_0)\subseteq \neg \varphi_0(\mathfrak{C})$ or $\pi_E^{-1}(e^0_1)\subseteq \neg \varphi_1(\mathfrak{C})$. For example, $\pi_E^{-1}(e^0_0)\subseteq \neg \varphi_0(\mathfrak{C})$. Set $I^0:=I\cup \{e^0_0\}$; this is an indiscernible sequence. 
Let $J_0^0$ be an end segment of $J_0$ not containing  $\mathfrak{d}_0$, and $J_1^0:=J_1$. 
By the same argument as in the above claim, we get
		
		\begin{claim}
		For every two cuts $\mathfrak{d}^1_0,\mathfrak{d}^1_1$ in $I^0$ respectively interior to $J_0^0,J_1^0$ we can find hyperimaginaries $e^1_0$ filling $\mathfrak{d}^1_0$ and $e_1^1$ filling $\mathfrak{d}^1_1$ (seen as cuts in $I^0$) such that $\tp(e^1_0,e^1_1/I_0)=\tp(d_0,d_1/I_0)$.
		\end{claim} 

Fix $\mathfrak{d}^1_0,\mathfrak{d}^1_1$ as in the claim, and choose $e^1_0, e^1_1$ provided by the claim. Then $\pi_E^{-1}(e^1_0)\subseteq \neg \varphi_0(\mathfrak{C})$ or $\pi_E^{-1}(e^1_1)\subseteq \neg \varphi_1(\mathfrak{C})$. For example,  $\pi_E^{-1}(e^1_1)\subseteq \neg \varphi_1(\mathfrak{C})$. Set $I^1 := I^0 \cup \{e^1_1\}$; this is again an indiscernible sequence. 
Let $J_0^1: =J_0^0$, and $J_1^1$ be an end segment of $J_1^0$ not containing $\mathfrak{d}_1^1$. 

Iterating this process $\omega$ times, we get a sequence $(\varepsilon_k)_{k<\omega}$ of 0's and 1's and a sequence of hyperimaginaries $(e^k_{\varepsilon_k})_{k<\omega}$ such that $I\cup \{ e^k_{\varepsilon_k}\}_{k<\omega}$ is indiscernible, the  $e^k_{\varepsilon_k}$'s with $\varepsilon_k=0$ fill pairwise distinct cuts in $J_0$, the  $e^k_{\varepsilon_k}$'s with $\varepsilon_k=1$ fill pairwise distinct cuts in $J_1$, and $\pi_E^{-1}(e_{\varepsilon_k}^k)\subseteq \neg \varphi_{\varepsilon_k}(\mathfrak{C})$ for all $k<\omega$. W.l.o.g. $\varepsilon_k=0$ for all $k<\omega$.
		
		Finally, by Ramsey's theorem and compactness, we can find an indiscernible sequence of representatives of the hyperimaginaries from the indiscernible sequence $J_0\cup \{ e_0^k\}_{k<\omega}$. In this way, we have produced an indiscernible sequence for which $\varphi_0(x_0)$ has infinite alternation rank, which contradicts NIP.
	\end{proof}

\begin{teor}\label{teordistal}
If $(a_i)_{i \in \mathcal{I}}$ is a (dense) distal sequence of tuples from $\C^\lambda$, then $(a_i/E)_{i \in \mathcal{I}}$ is a distal sequence of hyperimaginaries.	Thus, if $T$ is distal, then $T^{heq}$ is distal.
\end{teor}

\begin{proof}
The fact that the first part implies the second follows from the observation that for any indiscernible sequence of hyperimaginaries we can find an indiscernible sequence of representatives. So let us prove the first part.

Assume $I:=(a_i)_{i \in \mathcal{I}}$ is a (dense) distal sequence of tuples from $\C^\lambda$, and let $I'=(a_i/E)_{i \in \mathcal{I}}$. Present $E$ as $\dcap_{t\in\mathcal{T}}R_t$ for some $\emptyset$-definable (not necessarily equivalence) relations $R_t$. Consider any distinct Dedekind cuts $\mathfrak{c}'_1$ and $\mathfrak{c}'_2$ of $I'$, 
say $\mathfrak{c}_1'$ is on the left from $\mathfrak{c}_2'$. They partition $I'$ into $I_1'$, $I_2'$, and $I_3'$. The cuts $\mathfrak{c}'_1$ and $\mathfrak{c}'_2$ induce Dedekind cuts $\mathfrak{c}_1$ and $\mathfrak{c}_2$ of $I$ which partition $I$ into $I_1$, $I_2$, and $I_3$. Take any $d_1$ and $d_2$ filling the cuts  $\mathfrak{c}'_1$ and $\mathfrak{c}'_2$, respectively. Apply Lemma \ref{2.8'} to this data  (taking left-polarization of $\mathfrak{c}_1'$ and $\mathfrak{c}_2'$) and $A$ being the set of all coordinates of all tuples $a_i$, $i \in \mathcal{I}$. This yields $d_1'=e_1'/E$ and $d_2'=e_2'/E$ satisfying the conclusion of Lemma \ref{2.8'}.

	\begin{claim}
		For every $i \in \{1,2\}$ there is  $b_i$ filling the cut $\mathfrak{c}_i$ such that $\pi_{E}(b_i)=d_i'$.
	\end{claim}
\begin{proof}[Proof of the claim] It is enough to consider $i=1$.
	By compactness, it suffices to show that for any formula $\varphi(x_1,\dots, x_n)\in \EM(I)$, $i_1<\dots <i_{k-1}\in \mathcal{I}_1$, and $i_{k+1}<\dots <i_{n}\in \mathcal{I}_2+\mathcal{I}_3$ there is $b_1$ such that $\models \varphi(a_{i_1},\dots, a_{i_{k-1}},b_1,a_{i_{k+1}},\dots,a_{i_n})$ and $\pi_{E}(b_1)=d_1'$. By compactness, it is enough to show that for every $t\in \mathcal{T}$ 
	$$\models \exists y(yR_te_1' \wedge \varphi(a_{i_1},\dots, a_{i_{k-1}},y,a_{i_{k+1}},\dots,a_{i_n}) ).$$ 
	Assume this fails. Choosing $t'\in \mathcal{T}$ such that $R_{t'}\circ R_{t'}\subseteq R_t$, we get $$\pi_{E}^{-1}(d_1')\subseteq\neg(\exists y)( yR_{t'}x\wedge \varphi(a_{i_1},\dots, a_{i_{k-1}},y,a_{i_{k+1}},\dots,a_{i_n}) )(\mathfrak{C})=:\theta(\mathfrak{C}).$$
	By the choice of $d_1'$, for $a_j/E$ from a co-final fragment of the left part of $\mathfrak{c}_1$ we have $\pi_{E}^{-1}(a_j/E)\not\subseteq \neg\theta(\mathfrak{C})$. So, for any of these indices $j$, there is $a'E a_j$ such that $$ \models \neg\exists y ( y R_{t'} a' \wedge \varphi(a_{i_1},\dots, a_{i_{k-1}},y,a_{i_{k+1}},\dots,a_{i_n}) ).$$ 
As the set of such indices $j$ is co-final in $\mathcal{I}_1$, we can find such an index $j \in \mathcal{I}_1$ with $j>i_{k-1}$. Then $y:=a_j$ contradicts the last formula (by the indiscernibility of $I$). 	
\end{proof}

By the distality of $I$, the sequence $I_1+b_1+I_2+b_2+I_3$ is indiscernible. Hence, the sequence $I_1'+d_1'+I_2'+d_2'+I_3'$  is also indiscernible. On the other hand, by our choice of $d_1',d_2'$, we know that $d_1d_2\equiv_{I'}d_1'd_2'$. Thus, the sequence $I_1'+d_1+I_2'+d_2+I_3'$ is indiscernible, too. As $d_1,d_2$ were arbitrary, we conclude that $I'$ is a distal sequence.
\end{proof}

\begin{cor}\label{corollary: stable iff bounded}
	For a distal theory $T$, a hyperdefinable set $X/E$ is stable if and only if $E$ is a bounded equivalence relation. In particular, for a group $G$ $\bigwedge$-definable in a distal theory, $G^{st}=G^{00}$.
\end{cor}

\begin{proof}
	If $E$ is bounded, then $X/E$ is stable (as each indiscernible sequence in $X/E$ is constant). To prove the other implication, assume that $X/E$ is stable. 
Since distality is preserved under naming parameters, w.l.o.g. both $X$ and $E$ are $\emptyset$-type-definable. If $X/E$ is not bounded, taking a very long sequence of pairwise distinct elements of $X/E$, by extracting indiscernibles, there exists a dense indiscernible sequence of pairwise distinct elements of $X/E$. By stability and Theorem \ref{equivteor}, this sequence is totally indiscernible.  Since non-constant, totally indiscernible sequences are not distal, we get a contradiction with the distality of $T^{heq}$ (which we have by Theorem \ref{teordistal}).
\end{proof}

One could give a short direct (i.e. not using Theorem \ref{teordistal}) proof of Corollary \ref{corollary: stable iff bounded}. However, we find it very natural to see Corollary \ref{corollary: stable iff bounded} as an easy consequence of  Theorem \ref{teordistal} which in turn is a fundamental result concerning distality and hyperimaginaries.

	\section{An example of $G^{st,0}\neq G^{st}\neq G^{00}=G$}\label{section: main example}

	Our objective is to find a definable group $G$ in a NIP theory $T$ satisfying $G^{st,0}\neq G^{st}\neq G^{00}=G$. In this section, we will change the notation: the group interpreted in the monster model will be denoted by $G^*$ instead of $G$.

Consider the structure $\mathcal{M}:=(\mathbb{R},+,I)$, where $I:=[0,1]$. Let $T:=\Th(\mathcal{M})$ and $G:=(\mathbb{R},+)$. 
Let $\mathcal{M}^*=(\mathbb{R}^*,+,I^*) \succ \mathcal{M}$ be a monster model 
($\kappa$-saturated and strongly $\kappa$-homogeneous for large $\kappa$)
which expands to a monster model  $(\mathbb{R}^*,+,\leq ,1) \succ (\mathbb{R},+,\leq,1)$ (with the same $\kappa$), and $G^*:=(\R^*,+)$. Denote by $\mu$ the subgroup of infinitesimals, i.e. $\bigcap_{n \in \mathbb{N}^+} [-1/n,1/n]^*$. 


Some observations below may follow from more general statements in the literature, but we want to be self-contained and as elementary as possible in the analysis of this example.

\begin{prop}\label{proposition: NIP and unstable}
	T has NIP and is unstable.
\end{prop} 
\begin{proof}
	The structure $\mathcal{M}$ is the reduct of the o-minimal structure $(\mathbb{R},+,\leq ,1)$, hence $T$ has NIP. 
	
 Note that for $\varepsilon\in(0,\frac{1}{2})$ we can write the interval $[-\varepsilon,\varepsilon]$ as $(I-\varepsilon)\cap (I+(\varepsilon -1))$. Hence, the formula $\varphi(x,y):= \text{``}x\in I-y \wedge x\in I+(y-1)\text{''}$ has SOP.
\end{proof}

\begin{remark}\label{1 in dcl}
	$0,1 \in \dcl^{\mathcal{M}}(\emptyset)$.
\end{remark}
\begin{proof}
$0 \in \dcl^{\mathcal{M}}(\emptyset)$ as the neutral element of $(\mathbb{R},+)$. To see that $1 \in \dcl^{\mathcal{M}}(\emptyset)$, note that $1$ is defined by the formula $\varphi(x):=\text{``}x\in I \wedge \forall y (y\in I \implies (x-y)\in I)\text{''}$.
\end{proof}

\begin{lema} \label{lemma: subgroups of G}
		\begin{enumerate}
			\item The only invariant subgroups of $G^*$ are: $\{0\}$, the subgroups of $\mathbb{Q}$, the subgroups of the form $\mu+R$ where $R$ is a subgroup of $\mathbb{R}$, and $G^*$.
			\item The only $\emptyset$-$\bigwedge$-definable subgroups of $G^*$ are $\{0\}$, $\mu$, and $G^*$.
			\item The only definable (over parameters) subgroups of $G^*$ are $\{0\}$ and $G^*$.
		\end{enumerate}
\end{lema}

\begin{proof}
	$(1)$ 
By q.e. in $(\R,+, \leq,1)^*$, Remark \ref{1 in dcl}, and the fact that the order restricted to any interval $[-r,r]$ (where $r \in \mathbb{N}^+$) is $\emptyset$-definable in $\mathcal{M}$ (see Lemma \ref{proposition: interdefinability}), the following holds in $\mathcal{M}^*$: 
\begin{enumerate}[label=(\roman*)]
\item All elements $a>\mathbb{R}$ have the same type over $\emptyset$.
\item 
$\dcl^{\mathcal{M}}(\emptyset)=\mathbb{Q}$.
\item 
For any $a\in \mathbb{R}\setminus\mathbb{Q}$, $a+\mu$ is the set of all realizations of a type in $S_1(\emptyset)$.
\item For any $a \in \mathbb{Q}$,  all the elements $a+h$, where $h$ ranges over positive infinitesimals, form the set of realizations of a type in $S_1(\emptyset)$; and the same is true for all elements $a-h$. 
\end{enumerate}
This easily implies that the groups in the lemma are indeed invariant. 

For the converse, let $H$ be an invariant subgroup. If it contains some $a>\mathbb{R}$ or $a < \mathbb{R}$, then $H=G^*$ by (i). So suppose that $H \subseteq \mathbb{R} + \mu$. 
If $H$ contains an element from $a + \mu$ for some $a \in \mathbb{R} \setminus \mathbb{Q}$, then, by (iii), it contains $a + \mu$ and so $\mu$ as well; thus, $H$ is of the form $\mu + R$, where $R$ is a subgroup of $\mathbb{R}$.
If $H$ contains some element $a+h$ with $a\in \mathbb{Q}$ and $h$ a positive infinitesimal, then, by (iv),  $H$ contains all the elements of that form. Since $H$ is a group, it contains $\mu$ (because we can subtract any two elements $a+h$, $a+h'$); thus, $H$ is again of the form $\mu + R$, where $R$ is a subgroup of $\mathbb{R}$.  If $H$ contains some element of the form $a-h$ with $a\in \mathbb{Q}$ and  $h$ a positive infinitesimal, we proceed in an analogous manner. The only remaining case is that $H$ is a subgroup of $\mathbb{Q}$.

	$(2)$ 
A $\emptyset$-$\bigwedge$-definable subgroup $H$ either contains some $a>\mathbb{R}$, in which case $H=G^*$, or the type defining $H$ implies the formula $x\leq n$ for some $n\in\mathbb{N}$. This implies that $H$ is contained in $\mu$, so, by (1), either $H=\{0\}$ or $H=\mu$.

	$(3)$ 
By o-minimality of $(\R,+, \leq,1)^*$, any definable subgroup (over parameters) $H$ of $G^*$ is a finite union of points and intervals, so the conclusion easily follows. 	
\end{proof}


\begin{cor}
	$G^*=G^{*0}=G^{*00}=G^{*000}$
\end{cor}

\begin{proof}
Since $G^*\geq G^{*0}\geq G^{*00}\geq G^{*000}$, it is enough to show that $G^{*000}=G^*$. But this follows from Lemma \ref{lemma: subgroups of G}(1), as the index of $\mathbb{R} + \mu$ in $G^*$ is unbounded.
\end{proof}

\begin{cor}
	${G^*}^{st,0}=G^*$
\end{cor}

\begin{proof}
It follows directly from Lemma \ref{lemma: subgroups of G}(3) and Proposition \ref{proposition: NIP and unstable}.
\end{proof}

		Let $\mathcal{N}:=(\mathbb{R}, +,-, R_r)_{r\in \mathbb{N}^+}$, where $R_r(x,y)$ holds if and only if $0\leq y-x\leq r$. Let $T':=\Th(\mathcal{N})$. 
	
	\begin{remark}\label{R} The family $(R_r)_{r\in \mathbb{N}^+}$ satisfy the following conditions:
		\begin{enumerate}
			\item $R_r(x,y)\iff R_r(0, y-x)$;
			\item $R_r(x,y)\iff R_r(-y,-x)$ ;
			\item $R_r(x,y)\iff R_{nr}(nx,ny)$ (where $n \in \mathbb{N}^+$);
			\item $R_r(x,y)\iff R_1(x,y) \vee R_1(x,y-1)\vee \dots \vee R_1(x,y-(r-1)) $.
		\end{enumerate}
	\end{remark}

	\begin{lema}\label{proposition: interdefinability}
		The structures $\mathcal{M}$ and $\mathcal{N}$ are interdefinable over $\emptyset$.
	\end{lema}

\begin{proof}
	 $\mathcal{M}$ is definable over $\emptyset$ in $\mathcal{N}$, because $x\in[0,1]$ iff $R_1(0,x)$ iff $R_1(x,2x)$.
	
	To see that $\mathcal{N}$ is definable over $\{1\}$ in $\mathcal{M}$, note that the function $-$ can be defined using $+$ as usual, $R_1(x,y)\iff y-x\in [0,1]$, and then use the last property in Remark \ref{R} to conclude that  all $R_r$, $r\in\mathbb{N}^+$, are definable over $\{1\}$ in $\mathcal{M}$.
Since $1\in\dcl^{\mathcal{M}}(\emptyset)$, we conclude that $\mathcal{N}$ is definable over $\emptyset$ in $\mathcal{M}$.
\end{proof}
	
	The result above shows that the theory $T'$ also has NIP and is unstable.

\begin{prop}\label{proposition: qe}
$T'$ has quantifier elimination after expansion by the constant $1$.
\end{prop}

\begin{proof}
We argue by induction on the length of the formula. So the proof boils down to showing that a primitive formula $(\exists y) \varphi(y,\bar x)$ is $T'$-equivalent to a quantifier free formula, assuming that all shorter formulas are $T'$-equivalent to qf-formulas. Recall that $(\exists y) \varphi(y,\bar x)$ being primitive means that $\varphi(y,\bar x)$ is a conjunction of atomic formulas and negations of such, i.e. $\varphi(y,\bar x) = \bigwedge_{j=1}^m R^{\varepsilon_j}_{r_j}(t^l_j(y,\bar x), t^r_j(y,\bar x))$, where $\varepsilon_j \in \{\pm 1,\pm 2\}$, $r_j \in \mathbb{N}^+$, $t^l_j(y,\bar x)$ and $t^r_j(y,\bar x)$ are terms, and: $R_{r_j}^{-2}(t,z):= \neg R_{r_j}(t,z)$, $R_{r_j}^2(t,z):= R_{r_j}(t,z)$, $R_{r_j}^{-1}(t,z): = \neg (t=z)$, $R_{r_j}^1(t,z): = (t=z)$.  

Using $+$, $-$, multiplying by suitable integers, Remark \ref{R}, and induction hypothesis, we can assume  that there is an integer $n \ne 0$  such that for every $j$: either $t^l_j(y,\bar x) =0$ and $t^r_j(y,\bar x) = ny - t_j(\bar x)$, or $t^l_j(y,\bar x) = ny - t_j(\bar x)$ and  $t^r_j(y,\bar x)=0$.

If some $\varepsilon_j=1$, one gets $ny= t_j(\bar x)$ and the quantifier $\exists y$ can be eliminated. So assume that all $\varepsilon_j \ne 1$. If additionally all $\varepsilon_j \ne 2$, then $(\exists y) \varphi(y,\bar x)$ is $T'$-equivalent to $\top$. So assume that some $\varepsilon_j= 2$, e.g. $\varepsilon_1=2$. Then $R^{\varepsilon_1}_{r_1}(t^l_1(y,\bar x), t^r_1(y,\bar x))$ either says (in $\mathcal{N}$) that $ny \in [t_1(\bar x) -r_1, t_1(\bar x)]$, or that $ny \in [t_1(\bar x), t_1(\bar x) +r_1]$. Suppose the latter case holds.  Consider all (finitely many) possibilities taking into account: 
\begin{itemize}
\item which terms $t_j(\bar x)-r_j, t_j(\bar x), t_j(\bar x) +r_j$ (for $j\in \{2,\dots,m\}$) belong to $[t_1(\bar x) -r_1, t_1(\bar x)]$; 
\item for those which belong to this interval, how they are ordered by $R_{r_1}$;
\item for $\varepsilon_j \ne -1$, writing  $R^{\varepsilon_j}_{r_j}(t^l_j(y,\bar x), t^r_j(y,\bar x))$ as ``$ny \in I$'' or as ``$ny \notin I$'', where $I:= [t_j(\bar x) -r_j, t_j(\bar x)]$ or $I:=[t_j(\bar x), t_j(\bar x) +r_j]$, we should specify which of the terms $t_1(\bar x)$ and $t_1(\bar x) +r_1$ belong to $I$.
\end{itemize}

Each of these possibilities is clearly a qf-definable condition on $\bar x$ (using finitely many integers, but they are terms, as 1 was added to the language). On the other hand, $(\exists y) \varphi(y,\bar x)$ is $T'$-equivalent to the disjunction of some subfamily of these conditions (by a simple combinatorics on intervals). Therefore, $(\exists y) \varphi(y,\bar x)$ is $T'$-equivalent to a qf-formula.
\end{proof}

\begin{prop}\label{proposition: stable quotient}
	The quotient $G^*/\mu$ is stable.
\end{prop}

\begin{proof}

By Theorem \ref{equivteor}, it is enough to show that for every $A \subseteq \mathcal{M}^*$ with $|A| \leq \mathfrak{c}$ we have $|S_{G^*/\mu}(A)| \leq \mathfrak{c}$. 

By Lemma \ref{proposition: interdefinability}, $\mathcal{M}^*$ can be treated as an elementary extension of $\mathcal{N}$.

Consider an arbitrary set $A$ as above. Put $V_A:=\Lin_{\mathbb{Q}}(A\cup \mathbb{R})$ and $\tilde{V}_A:=V_A+\mu$. First, note that for any $a,a'\in G^*$, if $a-a'\in \mu$, then trivially $\tp((a+\mu)/A)=\tp((a'+\mu)/A$). Now, consider any $a\in G^*\setminus \tilde{V}_A$. Then $a$ satisfies the formulas $$kx\neq t$$ and $$\neg R_r(t,kx)$$ for all $k \in \mathbb{Z} \setminus \{0\}$, $r\in \mathbb{N}^+$, and $t\in V_A$. By Remark \ref{R}, Lemma \ref{proposition: interdefinability} and Proposition \ref{proposition: qe}, these formulas completely determine $\tp(a/V_A)$. Hence, any $a,a'\in G^*\setminus \tilde{V}_A$ have the same type over $A$. Therefore, $|S_{G^*/\mu}(A)| \leq |\tilde{V}_A/\mu| +1 \leq|V_A| +1 = \mathfrak{c}$.	
\end{proof}

\begin{cor}
	$\mu= G^{*st}$.
\end{cor}

\begin{proof}
	It follows from Lemma \ref{lemma: subgroups of G}(2), Proposition \ref{proposition: stable quotient}, and Proposition \ref{proposition: NIP and unstable}.
\end{proof}
To summarize, we have proved that $G^*$ is a $\emptyset$-definable group in a monster model of a NIP theory such that ${G^*}^{st,0}\neq {G^*}^{st}\neq {G^*}^{00}=G^*$.

\section{How to construct examples with $G^{st} \ne G^{st,0}$?}\label{section 5}

Our context is that $G$ is a $\emptyset$-definable group in a monster model $\C$ of a complete theory $T$ with NIP.

\begin{prop}\label{proposition: based on Hrushovski}
If $G^{st}=G^{st,0}$, then any $\bigwedge$-definable subgroup $H$ of $G$ with $G/H$ stable is an intersection of definable subgroups.
\end{prop}

\begin{proof}
 Assume that $G^{st}=G^{st,0}=\dcap_{j\in J}G_j$, where $G_j$ are definable groups. Since $G/G^{st,0}$ is a stable hyperdefinable group, by \cite[Remark 2.5(iv)]{MR3796277}, the intersection of all the conjugates of each $G_j$ is a bounded subintersection, and so it is a finite subintersection (by compactness and definability of $G_j$). Replacing each $G_j$ by such a finite subintersection, we can assume that all the $G_j$'s are normal subgroups of $G$. 
Hence, for every $j\in J$ we have  $G_j\leq H\cdot G_j\leq G$ and  
 $$\bigslant{H\cdot G_j}{G_j}\leq \bigslant{G}{G_j}.$$
 Using Hrushovski's theorem (see \cite[Ch. 1, Lemma 6.18]{pillay1996geometric}) inside the (definable) stable group $G/G_j$, we get $$\bigslant{H\cdot G_j}{G_j} = \dcap_{i\in I_j} \bigslant{ K^i_j}{G_j},$$ for some definable subgroups $K_j^i$ of $G$ such that $K_j^i\cdot G_j=K^i_j$ for all $i \in I_j$.

Since $G/H$ is assumed to be stable, $G^{st} \leq H$. Thus,
  $$H=H \cdot \dcap_{j\in J}G_j= \dcap_{j\in J}H\cdot G_j=\dcap_{j\in J}\dcap_{i\in I_j}K^i_j.$$
\end{proof}

If we do not require that $G^{00} = G^0$, then it is easy to find
examples where $G^{00} \ne G^{st} \ne G^{st,0}$; that is why  \cite{MR3796277}
required in this problem also $G^{00}=G$. However, the requirement $G^{00} =
G^0$ seems sufficiently interesting. The next proposition yields the whole class of examples where $G^{00} \ne G^{st} \ne G^{st,0}$ (but without the requirement that $G^{00}=G^0$).

\begin{prop}\label{4.10}
Let $G$ be definably isomorphic to a definable semidirect
product of definable groups $H$ and $K$ (symbolically, $G
\cong_{def} H \ltimes K$) such that $H^{00}\ne H^0$ and $K^{st} \ne
K^{00}$. Then $G^0 \ne G^{00} \ne G^{st} \ne G^{st,0}$.
\end{prop}

\begin{proof}
	W.l.o.g. $G = H \ltimes K$ and $H$, $K$, and the action of $H$ on $K$ are $\emptyset$-definable. Recall that by NIP, the 00-components exist (i.e. do not depend on the choice of parameters over which they are computed). Hence, $K^{00}$ is invariant under all definable automorphisms, in  particular under the action of $H$. So $G^{00} = H^{00} \ltimes K^{00}$ (e.g. by Corollary 4.11 in \cite{gismatullin2020bohr}).
	But $G^{st} \leq H^{00} \ltimes K^{st}$, because the map $(h,k)/ (H^{00} \ltimes K^{st}) \mapsto h/H^{00} \times k/K^{st}$ is an invariant bijection from 
	$G / (H^{00} \ltimes K^{st})$ to $H/H^{00} \times K/K^{st}$ and the last set is stable as a product of stable sets.
	 Thus, since $K^{st}$
	is a proper subgroup of $K^{00}$, we get that $G^{00} \ne G^{st}$.

	To see that $G^{st} \ne G^{st,0}$,  it is enough to note that $H^{00} \ltimes K \leq
	G$ and that  $H^{00} \ltimes K$ is not an
	intersection of definable groups (because $G/(H^{00} \ltimes K) \cong_{def}
	H/H^{00}$ is not profinite in the logic topology as $H^{00} \ne H^0$). Indeed, having this, since $G/(H^{00} \ltimes K)$  is bounded and so stable, by Proposition \ref{proposition: based on Hrushovski}, we conclude that  $G^{st} \ne G^{st,0}$.

	The fact that $G^{00} \ne G^{0}$ follows  from $H^{00}\ne H^0$, as $G^{0}= H^0 \ltimes K^0$.
\end{proof}

\begin{remark}
The assumption of Proposition \ref{4.10} is equivalent to saying
that $G$ has a definable, normal  subgroup $K$ with $K^{st} \ne
K^{00}$
and $(G/K)^{00} \ne (G/K)^{0}$ such that the quotient map $G \to
G/K$ has a section which is a definable homomorphism.
\end{remark}

The proof of Proposition \ref{4.10} can be easily modified to get the following variant.

\begin{remark}
The conclusion of Proposition \ref{4.10} remains true with the assumption ``$H^{00}\ne H^0$ and $K^{st} \ne
K^{00}$'' replaced by ``$H^{st}\ne H^{00}$ and $K^{00} \ne
K^{0}$''.
\end{remark}

One can find many examples satisfying the assumptions of Proposition \ref{4.10}.
For instance, take any group $H$ (definable in a monster model of a NIP
theory $T_1$) with $H^{00} \ne H^0$ (e.g. the circle group in the
theory of real closed fields) and any group $K$  (definable in a monster
model of a NIP theory $T_2$; where $T_1$ and $T_2$ are in disjoint
languages) with $K^{st} \ne K^{00}$ (e.g. $T_2$ is stable and $K$ is
infinite). Consider $T$ being the union of $T_1$ and $T_2$ living on two
disjoint sorts. Then $G:=H \times K$ satisfies the assumptions of Proposition \ref{4.10}
as a group definable in $T$.

One could still ask if it is possible to find examples satisfying the condition $G^0 =
G^{00} \ne G^{st} \ne G^{st,0}$ by finding a definable, normal subgroup
$K$ satisfying $K^{00}\ne K^0$ where $G \to G/K$ does not have a definable
section. 
However, there is no chance for this potential method to work for groups
of finite exponent, as for any torsion (equivalently finite exponent)
group $K$ definable in a monster model, we have $K^{00}=K^0$. This is
because $K/K^{00}$ is a compact torsion group, and such groups are known
to be profinite (see \cite[Theorem 8.20]{hewitt2013abstract}). 


\begin{problem}
Construct $G$ of finite exponent with $G^{00} \ne G^{st} \ne G^{st,0}$. (The equality $G^{0}=G^{00}$ always holds by the fact at the end of the last paragraph.)
\end{problem}

In the final part of this section, we describe how one could try to construct examples where  $G^{00} \ne G^{st} \ne G^{st,0}$. In fact, originally we used this approach to find the example in Section \ref{section: main example}. We will also point out a difference between the situation in the example from Section \ref{section: main example} and the finite exponent case.

	\begin{prop}\label{4.1}
		The conditions $G^{st,0}\neq G^{st}$ and $G^{st}\neq G^{00}$ are equivalent to the existence of a $\bigwedge$-definable subgroup $H$ of  $G$ such that:
		\begin{enumerate}
			\item $H$ is a countable  intersection $\dcap_{n < \omega} D_n$ of definable subsets of $G$ satisfying $D_{n+1}D_{n+1}\subseteq D_{n}$ and symmetric (i.e. $D_{n}^{-1}=D_{n}$ and $e \in D_n$);
			\item $[G:H]$ is unbounded;
			\item $H$ is not an intersection of definable subgroups;
			\item $G/H$ is stable.
		\end{enumerate}
	\end{prop}
\begin{proof}
	If $G^{st}\neq G^{00}$, then, by compactness, $G^{st}=\dcap \{H: H \text{ satisfies (1), (2), (4)}\}.$ Hence, assuming additionally that $G^{st,0}\neq G^{st}$, at least one of those groups $H$ has to also satisfy condition $(3)$.
	
	Assume now that $H\leq G$ satisfies conditions $(1)$, $(2)$, $(3)$, and $(4)$. Since $G^{st} \leq H$, we get $G^{st}\neq G^{00}$. The fact that $G^{st,0}\neq G^{st}$ follows from Proposition \ref{proposition: based on Hrushovski}. 
\end{proof}

Note that assuming (1), the negation of (3) is equivalent to saying that for every $n<\omega$ there is $m>n$ and a definable subgroup $K$ of $G$ such that $D_m \subseteq H \subseteq D_n$.

\begin{remark}
	If we have a situation as in the last proposition, then the same holds for $G$ treated as a group definable in $(G,\cdot, (D_n)_{n<\omega})$.
\end{remark}

So an idea is to look for a group $G$ and a decreasing sequence $(D_n)_{n<\omega}$ of symmetric subsets of $G$ with $D_{n+1}D_{n+1} \subseteq D_n$ for all $i<\omega$, such that for $M:=(G,\cdot, (D_n)_{n<\omega}))$ and $G^*:= G(\C)$ (where $\C =(G^*,+, (D_n^*)_{n<\omega}) \succ M$ is a monster model), the group $H:=\dcap_{n < \omega} D_n^*$ satisfies (1)-(4) from the last proposition (with $*$ added everywhere). In the example from Section \ref{section: main example}, $G := (\R,+)$ and as $D_n$ we can take $[-1/2^n,1/2^n]$. Then the $D_n$'s are definable in $(\R,+,[0,1])$,
hence $M$ is interdefinable with $(\R,+,[0,1])$, and so we focused on the latter structure. In the proof of stability of $G^*/\mu$ (see Proposition \ref{proposition: stable quotient}), for the counting argument to work it was important that $D_{n+1}$ is generic in $D_{n}$ (i.e. finitely many translates of $D_{n+1}$ cover $D_n$), 
as this guarantees that $\mu = \bigcap D_n^*$ has bounded index in the subgroup generated by $D_1^*$.  The next proposition shows that for abelian groups of finite exponent this genericity condition always fails.


\begin{prop}
	If $G$ is abelian of finite exponent, then there is no sequence $(D_n)_{n<\omega}$ of definable sets such that: \begin{enumerate}
		\item $D_n$ is symmetric and  $D_{n+1}+D_{n+1}\subseteq D_n$ for all $n<\omega$;
		\item $D_{n+1}$ is generic in $D_n$ for all $n<\omega$;
		\item $\dcap_{n<\omega} D_n$ is not an intersection of definable groups. 
	\end{enumerate}
\end{prop}
\begin{proof}
	Assume that there is such a sequence  $(D_n)_{n<\omega}$ of definable sets. Replacing $D_n$ by $D_{n+1}$ if necessary, we can assume that $D_0$ is an approximate subgroup (i.e. finitely many translates of $D_0$ cover $D_0+D_0$), because $D_1$ is an approximate subgroup by (1) and (2).
	We denote $D_0^{+n}:=D_0+\overset{n}{\dots}+D_0$. Then, $\langle D_0\rangle =\acup\limits_{n<\omega}D_0^{+n}$ is a $\bigvee$-definable group and, by (1), (2), and the assumption that $D_0$ is an approximate subgroup, we see that $\dcap_{k<\omega}D_k\leq \langle D_0\rangle$ is a $\bigwedge$-definable subgroup of bounded index. Hence, $$H:=\bigslant{\langle D_0\rangle}{\dcap_{k<\omega}D_k}$$ is a locally compact group with the logic topology (in which closed sets are defined as those whose preimages under the quotient map have type-definable intersections with all sets $D_0^{+n}$, $n<\omega$; see \cite[Lemma 7.5]{MR2373360}). 
Since $H$ is a torsion group, it follows from \cite[Theorem 3.5]{MR637201} that $H$ has a basis  $(H_i)_{i\in I}$ of neighbourhoods of the identity consisting of clopen subgroups. 
	Since each $\bigslant{D_n}{\dcap_{k<\omega}D_k}$ is a neighborhood of the identity, there is $H_n\subseteq \bigslant{D_n}{\dcap_{k<\omega}D_k}$ which is a clopen subgroup of $H$. Let $\pi: \langle D_0\rangle \to H$ be the quotient map. Then, $$\pi^{-1}[H_n]\subseteq D_n+\dcap_{k<\omega}D_k\subseteq D_n+D_n\subseteq D_{n-1}$$ is a $\bigwedge$-definable group. Since $\pi^{-1}[H_n]^c\cap D_{n-1}$ is also $\bigwedge$-definable, we deduce that $\pi^{-1}[H_n]$ is a definable group laying between $\dcap_{k<\omega} D_k$ and $D_{n-1}$. 
Since this is true for any $n>0$, we get a contradiction with $(3)$.
\end{proof}

The following corollary yields some hints on how an example of finite exponent could be constructed.

\begin{cor}
    	If $G$ is abelian of finite exponent, then the condition $G^{st,0}\neq G^{st}\neq G^{00}$ is equivalent to the existence of a sequence $(D_n)_{n<\omega}$ of definable sets such that:
    	\begin{enumerate}
		\item $D_n$ is symmetric and  $D_{n+1}+D_{n+1}\subseteq D_n$, for all $n<\omega$;
		\item $D_{n+1}$ is not generic in $D_n$ for all $n<\omega$;
		\item $\dcap_{n<\omega} D_n$ is not an intersection of definable groups;
		\item $[G: \dcap_{n<\omega} D_n]$ is unbounded;
		\item $\bigslant{G}{\dcap_{n<\omega} D_n}$ is stable.
	\end{enumerate}
\end{cor}
\begin{proof}
From Proposition \ref{4.1}, we obtain that the condition $G^{st,0}\neq G^{st}\neq G^{00}$ is equivalent to the existence of a sequence  $(D_n)_{n<\omega}$ satisfying  $(1)$, $(3)$, $(4)$, and $(5)$. Furthermore, by the previous proposition, such a sequence $(D_n)_{n<\omega}$ must contain an (infinite) subsequence satisfying $(2)$. 
\end{proof}


\begin{remark}
This section could be naturally generalized to the context of a $\bigwedge$-definable group $G$. This would require checking a few things, mainly that Hrushovski's theorem (i.e. \cite[Ch. 1, Lemma 6.18]{pillay1996geometric}) is valid for a stable $\bigwedge$-definable group (not necessarily living in a stable theory). We leave it to the reader.
\end{remark}

\printbibliography
\nocite{*}

\end{document}